\documentclass{amsart}
\usepackage{amsfonts,hyperref}

\newtheorem{theorem}{Theorem}
\theoremstyle{plain}
\newtheorem{acknowledgement}{Acknowledgement}

\newtheorem{corollary}{Corollary}

\newtheorem{remark}{Remark}

\numberwithin{equation}{section}
\theoremstyle{theorem}

\begin{document}
\title[Cauchy Polynomials with $q$ Parameter in Terms of $r$-Whitney Numbers]{
Explicit Formulae of Cauchy Polynomials with a $q$ Parameter in Terms of $r$-Whitney Numbers}
\author{F. A. Shiha}
\address{Department of Mathematics, Faculty of Science,
Mansoura University, 35516 Mansoura, EGYPT.}
\email{fshiha@yahoo.com, fshiha@mans.edu.eg}

\begin{abstract}
The Cauchy  polynomials with a $q$ parameter were recently defined, and several arithmetical properties were studied. In this paper, we establish explicit formulae for computing the Cauchy  polynomials with a $q$ parameter in terms of $r$-Whitney numbers of the first kind. We also obtain several properties and combinatorial identities.
\end{abstract}

\maketitle

\bigskip AMS (2010) Subject Classification: 05A15, 05A19, 11B73, 11B75.


\section{Introduction}
\bigskip  The Cauchy polynomials of the first kind $c_n(z)$ are defined by
\begin{equation}
c_n(z)=\int_0^1 (x-z)_n \, dx,
\end{equation}
and the Cauchy polynomials of the second kind $\hat{c}_n(z)$ are defined by
\begin{equation}
\hat{c}_n(z)=\int_0^1 (-x+z)_n \, dx,
\end{equation}
where $(y)_n=\prod_{i=0}^{n-1}(y-i)$ is the falling factorial with $(y)_0=1$. The exponential generating function of these polynomials are
\begin{equation}
\sum_{n=0}^{\infty}c_n(z)\,\frac{t^n}{n!}=\frac{t}{(1+t)^z\,\ln(1+t)}.
\end{equation}
\begin{equation}
\sum_{n=0}^{\infty} \hat{c}_n(z)\,\frac{t^n}{n!}=\frac{t(1+t)^z}{(1+t)\,\ln(1+t)}.
\end{equation}
(see \cite{kom13, kam2013}). When $z=0$, $c_n(0)=c_n$ and $\hat{c}_n(0)=\hat{c}_n$ are the Cauchy numbers of the first and second kind (see \cite{comtet74, mer06, zhao09, kom013}).

Recently \cite{mezo16} obtained a representation of the integer values of Cauchy polynomials in terms of
$r$-Stirling numbers of the first kind $s_r(n,k)$ \cite{broder84}. For all integers $n, r \geq 0$,
\begin{equation}\label{E:c1m}
c_n(r)=\sum_{k=0}^n s_r(n+r,k+r)\,\frac{1}{k+1},
\end{equation}
\begin{equation}\label{E:c2m}
\hat{c}_n(-r)=\sum_{k=0}^n (-1)^k \, s_r(n+r,k+r)\,\frac{1}{k+1}.
\end{equation}
Given variables $y$ and $m$ and a positive integer $k$, define the generalized rising and falling
factorials of order $k$ with increment $m$ by
 \begin{equation*}
 [y|m]_k=\prod_{j=0}^{k-1}(y+jm),\qquad [y|m]_0=1
\end{equation*}
 \begin{equation*}
 (y|m)_k=\prod_{j=0}^{k-1}(y-jm),\qquad (y|m)_0=1.
\end{equation*}
Komatsu \cite{kom2013} introduced the Poly-Cauchy polynomials and numbers with a $q$ parameter, and the Cauchy polynomials and numbers with a $q$ parameter as special cases.

Let $q$ be a real number with $q\neq 0$, Komatsu \cite{kom2013} defined the Cauchy polynomials with a $q$ parameter of the first kind $c_n^q(z)$ by
\begin{equation}\label{E:komc1}
c_n^q(z)=\int_{0}^{1}(x-z|q)_n \, dx
\end{equation}
and the Cauchy polynomials with a $q$ parameter of the second kind $\hat{c}_n^q(z)$ by
\begin{equation}\label{E:komc2}
\hat{c}_n^q(z)=\int_{0}^{1}(-x+z|q)_n \, dx.
\end{equation}
The exponential generating functions are
\begin{equation}
\sum_{n=0}^{\infty}c_n^q(z)\frac{t^n}{n!}=(1+qt)^{\frac{-z}{q}}\sum_{k=0}^{\infty}\left(\frac{\ln(1+qt)}{q}\right)^k\:\frac{1}{k!}\:\frac{1}{k+1},
\end{equation}
\begin{equation}
\sum_{n=0}^{\infty}\hat{c}_n^q(z)\frac{t^n}{n!}=(1+qt)^{\frac{z}{q}}\sum_{k=0}^{\infty}\left(-\,\frac{\ln(1+qt)}{q}\right)^k\:\frac{1}{k!}\:\frac{1}{k+1}.
\end{equation}
If $z=0$, then $c_n^q(0)=c_n^q$ and $\hat{c}_n^q(0)=\hat{c}_n^q$ are the Cauchy numbers with $q$ parameter of the first and second kind, respectively.
If $q=1$, then $c_n^1(z)=c_n(z)$ and $\hat{c}_n^1(z)=\hat{c}_n(z)$.

The $r$-Whitney numbers of the first and second kind were introduced by Mez\"o \cite{mezo2010}.  For non-negative integers $n$ and $k$ with $0\leq k\leq n$, let $w(n,k)=w_{q,r}(n,k)$ denote the $r$-Whitney numbers of the first kind, which are defined by
\begin{equation}\label{E:wi1}
q^n(x)_n=\sum_{k=0}^n\:w(n,k)\:(qx+r)^k.
\end{equation}
Let $W(n,k)=W_{q,r}(n,k)$ denote the $r$-Whitney numbers of the second kind, which are defined by
\begin{equation}\label{E:wi2}
(qx+r)^n=\sum_{k=0}^{n}q^k\:W(n,k)\:(x)_k.
\end{equation}
Usually $r$ is taken to be a non-negative integer and $q$ a positive integer, but both may
also be regarded as real numbers \cite{shat2017}. The exponential generating function of $w(n,k)$ is given by \cite{mezo2010}
\begin{equation}\label{E:ewi1}
\sum_{n\geq k}w(n,k)\frac {t^n}{n!}=(1+qt)^{\frac{-r}{q}}\left(\frac{\ln(1+qt)}{q}\right)^k\:\frac{1}{k!},
\end{equation}

\section{\protect\bigskip Basic results}

Replace $x$ by $\frac{x-r}{q}$ in \eqref{E:wi1}, then the $r$-Whitney numbers of the first kind $w(n,k)$ are given by
\begin{equation}\label{E:w11}
(x-r|q)_n=\prod_{j=0}^{n-1}(x-r-jq)=\sum_{k=0}^{n}w(n,k)\:x^k, \qquad q \neq 0,
\end{equation}
Using \eqref{E:komc1}, we get the following theorem.
\begin{theorem}
The Cauchy polynomials with $q$ parameter of the first kind $c_n^q(r)$, $q \neq 0$ can be written explicitly as
\begin{equation}\label{D:cau1}
c_n^q(r)=\sum_{k=0}^{n}w(n,k)\:\frac{1}{k+1}.
\end{equation}
\end{theorem}

The first few polynomials are

$c_0^q(r)=1$,

$c_1^q(r)=-r+\frac{1}{2}$,

$c_2^q(r)=r^2+(q-1)r-\frac{1}{2}q+\frac{1}{3}$,

$c_3^q(r)=-r^3-\frac{3}{2}(2q-1)r^2+(-2q^2+3q-1)r+q^2-q+\frac{1}{4}$,

$c_4^q(r)=r^4+(6q-2)r^3+(11q^2-9q+2)r^2+(6q^3-11q^2+6q-1)r-3q^3+\frac{11}{3}q^2-\frac{3}{2}q+\frac{1}{5}$.

\begin{remark}
If $r=0$, then $c_{n}^q(0)=c_{n}^q$ are the Cauchy numbers with $q$ parameter of the first kind \cite{kom2013}
\[
c_{n}^q=\int_{0}^{1}(x|q)_n\,dx=\sum_{k=0}^n q^{n-k}\,s(n,k)\,\frac{1}{k+1},
\]
where $s(n,k)$ are the Stirling numbers of the first kind.

If $q=1$, we have $c_n^1(r)=c_n(r)$ and $w_{1,r}(n,k)$ are reduced to $s_r(n+r,k+r)$, and hence we obtain the explicit formula \eqref{E:c1m}.
\end{remark}

From \eqref{E:ewi1}, we can easily derive the exponential generating function of $c_n^q(r)$  as follows:
\begin{equation*}
 \begin{split}
 \sum_{n=0}^{\infty}c_n^q(r)\frac{t^n}{n!}& =\sum_{n=0}^{\infty}\,\sum_{k=0}^{n}w(n,k)\frac{1}{k+1}\:\frac{t^n}{n!}\\
 &=\sum_{k=0}^{\infty}\, \sum_{n=k}^{\infty}w(n,k)\frac{t^n}{n!}\:\frac{1}{k+1}\\
 &=(1+qt)^{\frac{-r}{q}} \sum_{k=0}^{\infty}\left(\frac{\ln(1+qt)}{q}\right)^k\:\frac{1}{k!}\:\frac{1}{k+1}\\
 &=(1+qt)^{\frac{-r}{q}}\sum_{k=0}^{\infty}\left(\frac{\ln(1+qt)}{q}\right)^{k+1}\:\frac{1}{(k+1)!}\:\frac{q}{\ln(1+qt)}\\
 &=\frac{q(1+qt)^{\frac{-r}{q}}}{\ln(1+qt)}\sum_{k=1}^{\infty}\left(\frac{\ln(1+qt)}{q}\right)^{k}\:\frac{1}{k!}\\
 &=\frac{q(1+qt)^{\frac{-r}{q}}}{\ln(1+qt)}\left((1+qt)^{\frac{1}{q}}-1 \right).
 \end{split}
 \end{equation*}

When $r=0$, we get the exponential generating function of $c_n^q$
\begin{equation*}
 \sum_{n=0}^{\infty}c_n^q \,\frac{t^n}{n!}=\frac{q}{\ln(1+qt)}\left((1+qt)^{\frac{1}{q}}-1 \right)
 \end{equation*}
According to \eqref{E:w11},
\begin{equation}
(-x-r|q)_n=\prod_{j=0}^{n-1}(-x-r-jq)=\sum_{k=0}^{n}w(n,k)\,(-1)^k\,x^k, \qquad q \neq 0.
\end{equation}
Using \eqref{E:komc1}, we get the following theorem.

\begin{theorem}
The Cauchy polynomials with $q$ parameter of the second kind $\hat{c}_n^q(r)$, $q \neq 0$  can be written explicitly as
\begin{equation}\label{D:cau2}
\hat{c}_n^q(-r)=\sum_{k=0}^{n}(-1)^k\,w(n,k)\:\frac{1}{k+1}.
\end{equation}
\end{theorem}

The first few polynomials are

$\hat{c}_0^q(r)=1$,

$\hat{c}_1^q(r)=r-\frac{1}{2}$,

$\hat{c}_2^q(r)=r^2-(q+1)r+\frac{1}{2}q+\frac{1}{3}$,

$\hat{c}_3^q(r)=r^3-\frac{3}{2}(2q+1)r^2+(2q^2+3q+1)r-q^2-q-\frac{1}{4}$,

$\hat{c}_4^q(r)=r^4-(6q+2)r^3+(11q^2+9q+2)r^2-(6q^3+11q^2+6q+1)r+3q^3+\frac{11}{3}q^2+\frac{3}{2}q+\frac{1}{5}$.

\begin{remark}
If $r=0$, then $\hat{c}_{n}^q(0)=\hat{c}_{n}^q$ are the Cauchy numbers with $q$ parameter of the second kind \cite{kom2013}
\[
\hat{c}_{n}^q=\int_{0}^{1}(-x|q)_n\,dx=\sum_{k=0}^n q^{n-k}\,s(n,k)\,\frac{(-1)^k}{k+1},
\]
\end{remark}
 Similarly, we can obtain the exponential generating function of $\hat{c}_n^q(r)$:
 \begin{equation}
 \begin{split}
\sum_{n=0}^{\infty}\hat{c}_n^q(r)\frac{t^n}{n!}&=(1+qt)^{\frac{r}{q}}\sum_{k=0}^{\infty}\left(-\,\frac{\ln(1+qt)}{q}\right)^k\:\frac{1}{k!}\:\frac{1}{k+1}\\
&=\frac{q(1+qt)^{\frac{r}{q}}}{\ln(1+qt)}\left(1-(1+qt)^{\frac{-1}{q}} \right).
\end{split}
\end{equation}
And
\begin{equation}
\sum_{n=0}^{\infty}\hat{c}_n^q\frac{t^n}{n!}=\frac{q}{\ln(1+qt)}\left(1-(1+qt)^{\frac{-1}{q}} \right).
\end{equation}
Replace $x$ by $\frac{x-r}{q}$ in \eqref{E:wi2}, then the $r$-Whitney numbers of the second kind $W(n,k)$ are given by
\begin{equation}
x^n=\sum_{k=0}^{n}W(n,k)(x-r|q)_k=\sum_{k=0}^{n}W(n,k)\prod_{j=0}^{k-1}(x-r-jq), \qquad q \neq 0.
\end{equation}
Thus, the relation between $c_n^{q}(r)$, $\hat{c}_n^{q}(r)$ and $W(n,k)$ can be obtained as follows:
\begin{equation}\label{E:Wc}
\sum_{k=0}^n W(n,k)\:c_k^{q}(r)=\int_{0}^{1}\sum_{k=0}^n\:W(n,k)(x-r|q)_k\:dx=\int_{0}^{1}x^n\:dx=\frac{1}{n+1}
\end{equation}
\begin{equation}\label{E:Wc2}
\sum_{k=0}^n W(n,k)\:\hat{c}_k^{q}(-r)=\int_{0}^{1}\sum_{k=0}^n\:W(n,k)(-x-r|q)_k\:dx=\int_{0}^{1}(-1)^n\,x^n\:dx=\frac{(-1)^n}{n+1}
\end{equation}
Cheon et al. \cite{cheon12} gave the following representation of $w(n,k)$ in terms of $s(n,k)$
\[
w(n,k)=\sum_{i=k}^n\:\binom{n}{i}\:(-1)^{n-i}\:q^{i-k}\:[r|q]_{n-i}\:s(i,k).
\]
Hence,

\begin{corollary}
The Cauchy polynomials $c_n^q(r)$ can be computed by using $s(n,k)$ as follows:
\begin{equation}
 \begin{split}
 c_n^q(r)& =\sum_{k=0}^n\:\sum_{i=k}^n \binom{n}{i}\:(-1)^{n-i}\:q^{i-k}\:[r|q]_{n-i}\:s(i,k)\:\frac{1}{k+1}\\
 &=\sum_{i=0}^n\:\sum_{k=0}^i \binom{n}{i}\:(-1)^{n-i}\:q^{i-k}\:[r|q]_{n-i}\:s(i,k)\:\frac{1}{k+1}.
 \end{split}
 \end{equation}
When $q=1$, we obtain the identity
\begin{equation}
c_n(r)=\sum_{i=0}^n\:\binom{n}{i}\:(-1)^{n-i}\:[r|1]_{n-i}c_i.
\end{equation}
\end{corollary}

The $r$-Whitney numbers $w_{q,r}(n,k)$ satisfy the following identity \cite{cheon12}.
\begin{equation}
w_{q,r+s}(n,k)=\sum_{j=k}^n (-1)^{n-j}\,\binom{n}{j}\,[r|q]_{n-j}\,w_{q,s}(j,k),
\end{equation}
hence, we obtain the following theorem.

\begin{theorem}
For $n \geq 0$, we have
\begin{equation}
c_n^q(r+s)=\sum_{j=0}^n (-1)^{n-j}\,\binom{n}{j}\,[r|q]_{n-j}\,c_j^q(s).
\end{equation}
\end{theorem}

\begin{proof}
\begin{equation*}
 \begin{split}
c_n^q(r+s)&=\sum_{k=0}^{n}w_{q,r+s}(n,k)\:\frac{1}{k+1}\\
&=\sum_{k=0}^n \, \sum_{j=k}^n (-1)^{n-j}\,\binom {n}{j}\,[r|q]_{n-j}\,w_{q,s}(j,k)\, \frac{1}{k+1}\\
&=\sum_{j=0}^n \, \sum_{k=0}^j (-1)^{n-j}\,\binom{n}{j}\,[r|q]_{n-j}\,w_{q,s}(j,k)\, \frac{1}{k+1}\\
&=\sum_{j=0}^n (-1)^{n-j}\,\binom{n}{j}\,[r|q]_{n-j}\,c_j^q(s).
\end{split}
\end{equation*}
\end{proof}
\begin{remark}
For $s=0$, we get
\begin{equation}
c_n^q(r)=\sum_{j=0}^n (-1)^{n-j}\,\binom{n}{j}\,[r|q]_{n-j}\,c_j^q.
\end{equation}
For $q=1$, we get
\begin{equation}
c_n(r+s)=\sum_{j=0}^n (-1)^{n-j}\,\binom{n}{j}\,[r|1]_{n-j}\,c_j(s).
\end{equation}
\end{remark}
\begin{acknowledgement}
The author thank Prof. Istv\'an Mez\"o for reading carefully the paper and giving helpful suggestions.
\end{acknowledgement}

\end{document}